\documentclass[12pt]{article}
\usepackage{amsmath,amssymb,amsbsy,amsfonts,amsthm,latexsym,
            amsopn,amstext,amsxtra,euscript,amscd,amsthm,url}

\usepackage[usenames, dvipsnames]{color}
\usepackage[utf8]{inputenc}
\usepackage[english]{babel}
\usepackage{mathrsfs}
\usepackage{caption}

\newtheorem{Thm}{Theorem}[section]
\newtheorem{lem}[Thm]{Lemma}

\DeclareMathOperator{\Reg}{Reg}

\newcommand{\Ocal}{O}

\DeclareMathOperator{\id}{id}
\global\long\def\epsilon{\varepsilon}

\begin{document}
\date{}

\title{Weighted averages of arithmetic functions over regular integers modulo $n$}
\author{Waseem Alass and Sumaia Saad Eddin}

\maketitle
{\def\thefootnote{}
\footnote{{\it Mathematics Subject Classification 2010: 11A25, 11N37.}} 

\begin{abstract}
We investigate weighted averages of arithmetic functions over regular integers modulo $n$, extending previous identities of Kiuchi and Matsuoka. We establish general transformation formulas for the associated partial sums corresponding to arbitrary arithmetic functions. These formulas provide a unified framework for deriving asymptotic results. As applications, we obtain explicit asymptotic formulas for weighted averages involving several classical multiplicative functions, including the identity, Möbius, divisor, and Jordan totient functions.
\end{abstract}

\maketitle
\section{Introduction and results}
Arithmetic functions involving greatest common divisors constitute a classical and active area of research in number theory. Among the most extensively studied examples is the $\gcd$-sum function (also known as Pillai's arithmetic function),
\[
P(n)=\sum_{k=1}^{n}\gcd(k,n),
\]
which was introduced by Pillai~\cite{Pillai} in 1933. He proved the remarkable identity
\[
P(n)=\sum_{d\mid n} d\,\phi\!\left(\frac{n}{d}\right),
\]
revealing a close relationship between the $\gcd$-sum function and the Euler totient function $\phi(n)$. Since then, numerous arithmetic, algebraic and asymptotic properties of $P(n)$ have been investigated. We refer the reader to the survey of T\'oth~\cite{T1} for a comprehensive account of this topic.

A natural extension of the classical $\gcd$-sum function is obtained by restricting the summation to the set of regular integers modulo $n$. Recall that an integer $k$ is called \emph{regular modulo $n$} if there exists an integer $x$ satisfying
\[
k^2x\equiv k\pmod n,
\]
that is, the residue class of $k$ is a regular element of the ring $\mathbb Z_n$. It is known that $k$ is regular modulo $n$ if and only if $\gcd(k,n)$ is a unitary divisor of $n$. Here, a divisor $d$ of $n$ is called \emph{unitary}, denoted by $d\mid\mid n$, whenever
\[
d\mid n
\quad\text{and}\quad
\gcd\!\left(d,\frac{n}{d}\right)=1.
\]

Let
\[
\Reg_n=\{k\in\mathbb N:1\le k\le n,\; k \text{ is regular }\pmod n\},
\]
and let $\rho(n)=|\Reg_n|$. Regular integers modulo $n$ have attracted considerable attention during the last two decades because of their rich arithmetic structure and their close connections with multiplicative and unitary arithmetic functions.

T\'oth~\cite{T2} introduced the $\gcd$-sum function over regular integers modulo $n$,
\[
\widetilde P(n)=\sum_{k\in\Reg_n}\gcd(k,n),
\]
and established the identity
\begin{equation}
\label{eq00}
\widetilde P(n)
=
\sum_{d\mid\mid n}
d\,\phi\!\left(\frac{n}{d}\right).
\end{equation}
Furthermore, he proved the asymptotic formula
\begin{equation}
\label{eq000}
\sum_{n\le x}\widetilde P(n)
=
\frac{x^2}{2\zeta(2)}
\left(K_1\log x+K_2\right)
+
\Ocal\!\left(
x^{3/2}
\exp\!\left(
-C\frac{(\log x)^{3/5}}
{(\log\log x)^{1/5}}
\right)
\right),
\end{equation}
where $C>0$ is an absolute constant, $\zeta(s)$ denotes the Riemann zeta function, 
\begin{equation}
\label{eq0000}
K_1
=
\prod_p
\left(
1-\frac{1}{p(p+1)}
\right),
\end{equation}
and
\begin{equation}
\label{eq00000}
K_2
=
K_1
\left(
2\gamma-\frac12
-
2\frac{\zeta'(2)}{\zeta(2)}
\right)
-
\sum_{n\ge1}
\frac{\mu(n)
(\log n-\alpha(n)+2\beta(n))}
{n\psi(n)}.
\end{equation}
Here, $\mu$ denotes the M\"obius function, $\gamma$ is Euler's constant, 
\[
\alpha(n)=
\sum_{p\mid n}
\frac{\log p}{p-1},
\qquad
\beta(n)=
\sum_{p\mid n}
\frac{\log p}{p^2-1},
\]
and
\[
\psi(n)
=
n
\prod_{p\mid n}
\left(
1+\frac1p
\right)
\]
is the Dedekind function.

More recently, Zhang and Zhai~\cite{Z} improved the error term in~\eqref{eq000}, assuming the Riemann Hypothesis, by proving
\[
\sum_{n\le x}\widetilde P(n)
=
\frac{x^2}{2\zeta(2)}
\left(K_1\log x+K_2\right)
+
\Ocal\!\left(x^{15/11+\varepsilon}\right),
\]
for every sufficiently small $\varepsilon>0$.

In another direction, Apostol and T\'oth~\cite{A.T} introduced multidimensional analogues of the function $\rho(n)$ and obtained elegant identities involving Bernoulli polynomials, the Gamma function and cyclotomic polynomials. More recently, Kiuchi and Matsuoka~\cite[Theorem~3.1]{K.M} proved that, for every arithmetic function $f$ and every fixed positive integer $r$,
\[
\frac1{n^r}
\sum_{k\in\Reg_n}
f(\gcd(k,n))k^r
=
\frac{f(n)}2
+
\frac1{r+1}
\sum_{m=0}^{ \lfloor r/2\rfloor}
\binom{r+1}{2m}
B_{2m}
\sum_{d\mid\mid n}
f\!\left(\frac nd\right)
\phi_{1-2m}(d),
\]
where $B_m$ denotes the $m$th Bernoulli number. For every integer $s$, the generalized Jordan totient function is defined by
\[
\phi_s(n)
=
\sum_{d\mid n}
d^s\mu\!\left(\frac{n}{d}\right)
=
n^s\prod_{p\mid n}
\left(1-\frac{1}{p^s}\right).
\]
In particular, $\phi_1=\phi$.

Although weighted identities over regular integers modulo $n$ have been established, their partial sums and the resulting asymptotic formulas for specific arithmetic functions have received comparatively less attention. The purpose of this paper is to derive general transformation formulas for these partial sums and to apply them to several classical multiplicative functions. Our first main result is the following.
\begin{Thm}
\label{Thm1}
 For any arithmetic function $f$, any sufficiently large positive number $x>2$ and fixed positive integer $r$, we have 
\begin{multline}
\label{eq1}
U_f(x):=\sum_{n\leq x}\frac{1}{n^{r+1}}\sum_{k\in \Reg_n}f(\gcd(k,n))k^r=\\
\frac{1}{2}\sum_{n\leq x}\frac{f(n)}{n}+\frac{1}{r+1}\sum_{m=0}^{\lfloor r/2\rfloor}\binom{r+1}{2m}B_{2m}\sum_{\substack{d\ell\leq x\\ \gcd(d, \ell)=1}}\frac{f(\ell)}{\ell}\frac{\phi_{1-2m}(d)}{d}. 
\end{multline}
If, in addition, $f(n)\neq 0$ for every $n\in\mathbb N$, then
\begin{multline}
\label{eq2}
V_f(x):=\sum_{n\leq x}\frac{1}{n^{r}f(n)}\sum_{k\in \Reg_n}f(\gcd(k,n))k^r=
\frac{x}{2}-\frac{\theta(x)}{2}-\frac{1}{4}\\+\frac{1}{r+1}\sum_{m=0}^{\lfloor r/2\rfloor}\binom{r+1}{2m}B_{2m}\sum_{\substack{d\ell\leq x\\ \gcd(d, \ell)=1}}\frac{f(\ell)\phi_{1-2m}(d)}{f(d\ell)}, 
\end{multline}
where $\theta(x)=x-\lfloor x \rfloor-1/2$.
\end{Thm}
The usefulness of Theorem~\ref{Thm1} is illustrated by the following applications, which provide explicit asymptotic formulas associated with the identity, M\"obius, divisor and second Jordan totient functions.
\begin{Thm}
\label{thm4}
Under the hypotheses of Theorem \ref{Thm1}, we have 
\begin{multline}
\label{eqcor11}
U_{\id}(x)=
\frac{K_1}{(r+1)\zeta(2)}x\log x\\+\left( \frac{1}{2}-\frac{K_1}{2(r+1)\zeta(2)}+\frac{K_2}{(r+1)\zeta(2)}+\frac{1}{r+1}\sum_{m=1}^{\lfloor r/2\rfloor}\binom{r+1}{2m}B_{2m}C_m\right)x\\+\Ocal_r\left( x^{1/2} \exp\left(-C\frac{(\log x)^{3/5}}{(\log \log x)^{1/5}}\right)\right),
\end{multline}

\begin{equation}
\label{eqcor12}
U_{\mu}(x)=
\frac{x}{(r+1)\zeta(2)}\prod_p\left(1-\frac{1}{p(p+1)}\right)+\Ocal_r\left( (\log x)^3\right), 
\end{equation}

\begin{equation}
\label{eqcor13}
U_{\tau}(x)=\\
\frac{x}{(r+1)\zeta(2)}\prod_{p}\left(1+\frac{(2p^2-1)p}{(p-1)^2(p+1)^3}\right)+\Ocal_r\left( (\log x)^5 \right), 
\end{equation}
where
\[
C_m
:=
\prod_p
\left(
1-\frac{(p-1)(p^{2m-1}-1)}
{p(p^{2m+1}-1)}
\right),
\]
and where $K_1$ and $K_2$ are defined by
\eqref{eq0000} and \eqref{eq00000}, respectively.
Moreover, we have
\begin{multline}
\label{eqcor14}
V_{\phi_2}(x)=\left(\frac{1}{2}+\frac{1}{r+1}\prod_{p}\left( 1+\frac{1}{(p+1)^2}\right)\right)x\\+\frac{x}{r+1}\sum_{m=0}^{\lfloor r/2\rfloor}\binom{r+1}{2m}B_{2m}\prod_p\left(1-\frac{p(p^{2m-1}-1)}{(p+1)(p^{2m+2}-1)}\right)+\Ocal_r\left( (\log x)^4\right).
\end{multline}
\end{Thm}
\section{Auxiliary results }
In this section, we collect several estimates that will be used in the proofs of the main results. Throughout, we write \[ \theta(x)=x-\lfloor x\rfloor-\frac12, \] and denote by $\tau^{\star}(n)$ the number of square-free divisors of $n$. We also write $\sigma(n)=\sigma_1(n)$.
\begin{lem}
\label{lem1} 
Let $t>1$ be an integer and let $x\geq 2$. Then
\begin{equation}
\label{eqlem11}
\sum_{\substack{n\leq x\\ \gcd(n, t)=1}}1=\frac{\phi(t)}{t}x-\sum_{d|t}\mu(d)\theta\left(\frac{x}{d}\right).
\end{equation}
Moreover,
\begin{equation}
\label{eqlem14}
\sum_{\substack{n\leq x\\ \gcd(n, t)=1}}\frac{\phi(n)}{n}=\frac{t\phi(t)}{\zeta(2)\phi_2(t)}x+\Ocal\left(\tau^{\star}(t)\log x\right).
\end{equation}
For every integer $s\geq 2$, we have
\begin{equation}
\label{eqlem15}
\sum_{\substack{n\leq x\\ \gcd(n, t)=1}}\frac{\phi_{-s}(n)}{n}=\Ocal\left(\frac{\phi_s(t)}{t^s}\log x+\frac{\sigma(t)}{t}\log x\right),
\end{equation}
whereas
\begin{equation}
\label{eqlem16}
\sum_{\substack{n\leq x\\ \gcd(n, t)=1}}\frac{\phi_{-1}(n)}{n}=
\Ocal\left( \frac{\phi(t)}{t}(\log x)^2+\frac{\phi(t)}{t}\log \log (3t)\log x+\frac{\phi(t)}{t}\log x+\tau (t)\log x\right).
\end{equation}
\end{lem}
\begin{proof}
Identity~\eqref{eqlem11} follows from the M\"obius inversion formula; see \cite[Lemma~2.1]{K.M}. The asymptotic formula~\eqref{eqlem14} is classical; see, for example, \cite{Sur} or \cite[Chapter~1, Section~I.24]{SMC}. To prove~\eqref{eqlem15} and~\eqref{eqlem16}, we use partial summation together with the estimates
\begin{equation*}
\sum_{\substack{n\leq x\\ \gcd(n, t)=1}}\phi_{-s}(n)=\Ocal\left( \frac{\phi_s(t)}{t^s}x+x\frac{\sigma(t)}{t}\right), \qquad (s\geq 2)
\end{equation*}
and 
\begin{equation*}
\sum_{\substack{n\leq x\\ \gcd(n, t)=1}}\phi_{-1}(n)=\Ocal\left( \frac{\phi(t)}{t}x\log x+x\frac{\phi(t)}{t}\log \log (3t)+x\tau (t)\right).
\end{equation*}
These estimates are given in \cite[Lemmas~2.2 and~2.3]{K.M}. 
\end{proof}
\begin{lem}
\label{lem2}
Let $x\geq 2$. There exists a constant $B>0$ such that
\begin{equation}
\label{eqlem23}
\sum_{n\leq x}\frac{\mu(n)}{n}=\Ocal\left(\exp \left(-B(\log x)^{\frac{3}{5}}(\log \log x)^{\frac{-1}{5}} \right)\right).
\end{equation} 
Furthermore,
\begin{equation}
\label{eqlem26}
\sum_{n\leq x}\frac{\phi(n)}{n^2}=\frac{1}{\zeta(2)}\log x+\frac{1}{\zeta(2)}\left(\gamma-\frac{\zeta'(2)}{\zeta(2)}\right)+\Ocal\left(\frac{\log x}{x}\right).
\end{equation}
We also have
\begin{equation}
\label{eqlem24}
\sum_{n\leq x}\frac{\tau(n)}{n}=\frac{1}{2}(\log x)^2+2\gamma\log x+(2\gamma-1)+\Ocal\left(x^{-285/416}(\log x)^{26947/8320}\right),
\end{equation}
\begin{multline}
\label{taustar}
\sum_{n\leq x}\frac{\tau^{\star}(n)}{n}=
\frac{1}{2\zeta(2)}(\log x)^2+\left(2\gamma+\frac{\zeta'(2)}{\zeta(2)}\right)\frac{\log x}{\zeta(2)}\\+\left(2\gamma-1+\frac{\zeta'(2)}{\zeta(2)}\right)\frac{1}{\zeta(2)}+\Ocal\left( x^{-1/2}\exp\left(A(\log x)^{3/5}(\log \log x)^{-1/5}\right)\right),
\end{multline}
and 
\begin{equation}
\label{WS3}
\sum_{n \leq x}\frac{\tau^2(n)}{n}=\frac{(\log x)^4}{4\pi^2}+\left(\frac{B_1}{3}+\frac{1}{\pi^2}\right)(\log x)^3+B_1(\log x)^2+C\log x+\Ocal(1),
\end{equation}
where $A>0$ and $B_1$, $C$ are constants, and $\zeta'$ denotes the derivative of the Riemann zeta function.
\end{lem}
\begin{proof}
Estimate~\eqref{eqlem23} is due to Jia~\cite{Jia}, while \eqref{eqlem26} is given in \cite[Lemma~2.1]{K2}. For~\eqref{eqlem24}, we use partial summation together with the classical divisor estimate
\begin{equation}
\label{eqlem27}
\sum_{n\leq x}\tau(n) = x\log x + (2\gamma-1)x + \Delta(x),
\end{equation}
where \[ \Delta(x) = \Ocal\!\left( x^{131/416} (\log x)^{26947/8320} \right) \] by a result of Huxley~\cite{H}.
Similarly, the estimates~\eqref{taustar} and~\eqref{WS3} follow by partial summation and
\begin{equation*}
\sum_{n\leq x}\tau^{\star}(n)=\frac{x}{\zeta(2)}\left(\log x +2\gamma-1+\frac{\zeta'(2)}{\zeta(2)}\right)+\Ocal\left( x^{1/2}\exp\left(-A_1(\log x)^{3/5}(\log \log x)^{-1/5}\right)\right) 
\end{equation*} 
and 
\begin{equation}
\label{tausquare}
\sum_{n \leq x}\tau^2(n)=\frac{1}{\pi^2}x (\log x)^3+B_1x(\log x)^2+Cx+D+\Ocal\left(x^{1/2+\varepsilon}\right),
\end{equation}
where $A_1>0$ and $B_1, C, D$ are constants and $\varepsilon>0$. The proof of these sums can be found in \cite{Sur1,W1} and \cite[Chapter~2, Section~II.13]{SMC}.
\end{proof}
\begin{lem}
For any sufficiently large positive number $x>2$, we have 
\label{lem3}
\begin{equation}
\label{WS1}
\sum_{\ell > x}\frac{\tau(\ell)}{\ell^2}=\Ocal\left(\frac{\log x}{x}\right),
\end{equation}
\begin{equation}
\label{WS2}
\sum_{\ell > x}\frac{\tau^2(\ell)}{\ell^2}=\Ocal\left(\frac{(\log x)^3}{x}\right),
\end{equation}
\begin{equation}
\label{WS4}
\sum_{\ell > x}\frac{\tau(\ell)}{\ell^3}=\Ocal\left(\frac{(\log x)}{x^2}\right).
\end{equation}
\end{lem}
\begin{proof}
Let \[ T(y)=\sum_{n\leq y}\tau(n). \] By~\eqref{eqlem27}, we have \[ T(y)\ll y\log y. \] Partial summation therefore gives \begin{align*} 
\sum_{\ell>x}\frac{\tau(\ell)}{\ell^2} &= -\frac{T(x)}{x^2} + 2\int_x^\infty\frac{T(t)}{t^3}\,dt \\ &\ll \frac{\log x}{x} + \int_x^\infty\frac{\log t}{t^2}\,dt \ll \frac{\log x}{x}, \end{align*} 
which proves~\eqref{WS1}.
Likewise, if 
\[
T_2(y)=\sum_{n\leq y}\tau^2(n), 
\] 
then~\eqref{tausquare} implies 
\[
T_2(y)\ll y(\log y)^3. 
\] 
Another application of partial summation yields 
\[
\sum_{\ell>x}\frac{\tau^2(\ell)}{\ell^2} \ll \frac{(\log x)^3}{x}, 
\] 
which proves~\eqref{WS2}. Finally, using again the bound $T(y)\ll y\log y$, we obtain 
\[
\sum_{\ell>x}\frac{\tau(\ell)}{\ell^3} = -\frac{T(x)}{x^3} + 3\int_x^\infty\frac{T(t)}{t^4}\,dt \ll \frac{\log x}{x^2}.
\] This proves~\eqref{WS4}.
\end{proof}

\section{Proofs}
\subsection{Proof of Theorem \ref{Thm1}}

The proof is based on the identity of Kiuchi and Matsuoka
\cite[Theorem~3.1]{K.M},
\[
\frac1{n^r}
\sum_{k\in\Reg_n}
f(\gcd(k,n))k^r
=
\frac{f(n)}2
+
\frac1{r+1}
\sum_{m=0}^{\lfloor r/2\rfloor}
\binom{r+1}{2m}
B_{2m}
\sum_{d\mid\mid n}
f\!\left(\frac nd\right)
\phi_{1-2m}(d).
\]

Substituting this identity into the definition of $U_f(x)$ gives
\[
U_f(x)
=
\frac12
\sum_{n\le x}
\frac{f(n)}n
+
\frac1{r+1}
\sum_{m=0}^{\lfloor r/2\rfloor}
\binom{r+1}{2m}
B_{2m}
\sum_{n\le x}
\frac1n
\sum_{d\mid\mid n}
f\!\left(\frac nd\right)
\phi_{1-2m}(d).
\]

Since
\[
d\mid\mid n
\iff
d\mid n,\qquad
(d,n/d)=1,
\]
we have
\begin{align*}
\sum_{n\le x}
\frac1n
\sum_{d\mid\mid n}
f\!\left(\frac nd\right)
\phi_{1-2m}(d)
&=
\sum_{n\le x}
\frac1n
\sum_{\substack{d\mid n\\ \gcd(d,n/d)=1}}
f\!\left(\frac nd\right)
\phi_{1-2m}(d)
\\
&=
\sum_{\substack{d\ell\le x\\ \gcd(d,\ell)=1}}
\frac{f(\ell)\phi_{1-2m}(d)}
{d\ell},
\end{align*}
which proves~\eqref{eq1}.

Assume now that $f(n)\neq0$ for every $n\in\mathbb N$. Proceeding exactly as above, we obtain
\[
V_f(x)
=
\frac12
\sum_{n\le x}1
+
\frac1{r+1}
\sum_{m=0}^{\lfloor r/2\rfloor}
\binom{r+1}{2m}
B_{2m}
\sum_{\substack{d\ell\le x\\ \gcd(d,\ell)=1}}
\frac{f(\ell)\phi_{1-2m}(d)}
{f(d\ell)}.
\]
Since
\[
\sum_{n\le x}1
=
x-\theta(x)-\frac12,
\]
we immediately obtain~\eqref{eq2}. This completes the proof.
\subsection{Proof of Theorem~\ref{thm4}}
\label{subsection2}

We begin with the asymptotic formula for $U_{\id}(x)$. Taking
$f=\id$ in~\eqref{eq1}, we obtain
\[
U_{\id}(x)
=
\frac12\sum_{n\leq x}1
+
\frac{1}{r+1}
\sum_{m=0}^{\lfloor r/2\rfloor}
\binom{r+1}{2m}B_{2m}
\sum_{\substack{d\ell\leq x\\ \gcd(d,\ell)=1}}
\frac{\phi_{1-2m}(d)}{d}.
\]
Since
\[
\sum_{n\leq x}1
=
x-\theta(x)-\frac12
\]
and $B_0=1$, it follows that
\begin{equation}
\label{Uid-decomposition}
U_{\id}(x)
=
\frac{x}{2}
-\frac{\theta(x)}{2}
-\frac14
+
\frac{1}{r+1}A_0(x)
+
\frac{1}{r+1}
\sum_{m=1}^{\lfloor r/2\rfloor}
\binom{r+1}{2m}B_{2m}A_m(x),
\end{equation}
where
\[
A_0(x)
:=
\sum_{\substack{d\ell\leq x\\ \gcd(d,\ell)=1}}
\frac{\phi(d)}{d}
\]
and, for $m\geq1$,
\[
A_m(x)
:=
\sum_{\substack{d\ell\leq x\\ \gcd(d,\ell)=1}}
\frac{\phi_{1-2m}(d)}{d}.
\]

We first estimate $A_0(x)$. By~\eqref{eq00},
\[
\frac{\widetilde P(n)}{n}
=
\sum_{d\mid\mid n}\frac{\phi(d)}{d},
\]
and hence
\[
A_0(x)
=
\sum_{n\leq x}\frac{\widetilde P(n)}{n}.
\]
Applying partial summation to the asymptotic formula~\eqref{eq000}, we obtain
\begin{equation}
\label{eqproof22}
A_0(x)
=
\frac{x}{\zeta(2)}
\left(
K_1\left(\log x-\frac12\right)+K_2
\right)
+
\Ocal\!\left(
x^{1/2}
\exp\!\left(
-C\frac{(\log x)^{3/5}}
{(\log\log x)^{1/5}}
\right)
\right).
\end{equation}

We next consider $A_m(x)$ for $m\geq1$. Reversing the order of summation gives
\[
A_m(x)
=
\sum_{d\leq x}
\frac{\phi_{1-2m}(d)}{d}
\sum_{\substack{\ell\leq x/d\\ \gcd(\ell,d)=1}}1.
\]
By~\eqref{eqlem11},
\begin{equation}
\label{eqproof23}
A_m(x)
=
x
\sum_{d\leq x}
\frac{\phi_{1-2m}(d)\phi(d)}{d^3}
-
\sum_{d\leq x}
\frac{\phi_{1-2m}(d)}{d}
\sum_{k\mid d}
\mu(k)
\theta\!\left(\frac{x}{kd}\right).
\end{equation}

The arithmetic function $\phi_{1-2m}(d)\phi(d)$ is multiplicative. Moreover,
\[
\sum_{d\geq1}
\frac{\left|\phi_{1-2m}(d)\right|\phi(d)}{d^3}
<\infty.
\]
Its Dirichlet series at $s=3$ therefore admits an absolutely convergent
Euler product. A calculation of the local factors gives
\[
\sum_{d\geq1}
\frac{\phi_{1-2m}(d)\phi(d)}{d^3}
=
\prod_p
\left(
1-
\frac{(p-1)(p^{2m-1}-1)}
{p(p^{2m+1}-1)}
\right).
\]
Since
\[
\left|\phi_{1-2m}(d)\right|
\leq 1
\qquad (m\geq1),
\]
we also have
\[
\sum_{d>x}
\frac{\left|\phi_{1-2m}(d)\right|\phi(d)}{d^3}
\ll
\sum_{d>x}\frac1{d^2}
\ll
\frac1x.
\]
Consequently,
\begin{equation}
\label{Am-main-series}
\sum_{d\leq x}
\frac{\phi_{1-2m}(d)\phi(d)}{d^3}
=
\prod_p
\left(
1-
\frac{(p-1)(p^{2m-1}-1)}
{p(p^{2m+1}-1)}
\right)+\Ocal\!\left(\frac1x\right).
\end{equation}

For the second term in~\eqref{eqproof23}, the bound
$|\theta(y)|\leq 1/2$ yields
\begin{align*}
\left|
\sum_{d\leq x}
\frac{\phi_{1-2m}(d)}{d}
\sum_{k\mid d}
\mu(k)
\theta\!\left(\frac{x}{kd}\right)
\right|
&\ll
\sum_{d\leq x}
\frac{\left|\phi_{1-2m}(d)\right|\tau(d)}{d}
\\
&\ll
\sum_{d\leq x}\frac{\tau(d)}{d}
\ll
(\log x)^2,
\end{align*}
where the last estimate follows from~\eqref{eqlem24}. Combining this
with~\eqref{eqproof23} and~\eqref{Am-main-series}, we obtain
\begin{equation}
\label{eqproof24}
A_m(x)
=
\prod_p
\left(
1-
\frac{(p-1)(p^{2m-1}-1)}
{p(p^{2m+1}-1)}
\right)x+\Ocal_m\!\left((\log x)^2\right).
\end{equation}
Finally, substituting~\eqref{eqproof22} and~\eqref{eqproof24}
into~\eqref{Uid-decomposition}, and observing that
\[
(\log x)^2
=
o\!\left(
x^{1/2}
\exp\!\left(
-C\frac{(\log x)^{3/5}}
{(\log\log x)^{1/5}}
\right)
\right),
\]
we obtain~\eqref{eqcor11}.\\

We next prove~\eqref{eqcor12}. Taking $f=\mu$ in~\eqref{eq1}, we obtain
\begin{equation}
\label{Umu-decomposition}
U_{\mu}(x)
=
\frac12\sum_{n\leq x}\frac{\mu(n)}{n}
+
\frac{1}{r+1}
\sum_{m=0}^{\lfloor r/2\rfloor}
\binom{r+1}{2m}B_{2m}L_m(x),
\end{equation}
where
\[
L_m(x)
:=
\sum_{\substack{d\ell\leq x\\ \gcd(d,\ell)=1}}
\frac{\phi_{1-2m}(d)}{d}
\frac{\mu(\ell)}{\ell}.
\]
By~\eqref{eqlem23},
\begin{equation}
\label{Umu-first-term}
\sum_{n\leq x}\frac{\mu(n)}{n}
=
\Ocal\!\left(
\exp\!\left(
-B(\log x)^{3/5}
(\log\log x)^{-1/5}
\right)
\right).
\end{equation}

We first consider the case $m=0$. Reversing the order of summation gives
\[
L_0(x)
=
\sum_{\ell\leq x}
\frac{\mu(\ell)}{\ell}
\sum_{\substack{d\leq x/\ell\\ \gcd(d,\ell)=1}}
\frac{\phi(d)}{d}.
\]
Applying~\eqref{eqlem14} with $t=\ell$ and $x$ replaced by $x/\ell$, we obtain
\begin{align}
L_0(x)
&=
\sum_{\ell\leq x}
\frac{\mu(\ell)}{\ell}
\left\{
\frac{\ell\phi(\ell)}
{\zeta(2)\phi_2(\ell)}
\frac{x}{\ell}
+
\Ocal\!\left(
\tau^{\star}(\ell)
\log\!\left(\frac{x}{\ell}\right)
\right)
\right\}
\nonumber\\
&=
\frac{x}{\zeta(2)}
\sum_{\ell\leq x}
\frac{\mu(\ell)\phi(\ell)}
{\ell\phi_2(\ell)}
+
\Ocal\!\left(
\log x
\sum_{\ell\leq x}
\frac{\tau^{\star}(\ell)}{\ell}
\right).
\label{WA}
\end{align}
The function $\frac{\mu(\ell)\phi(\ell)}{\ell\phi_2(\ell)}$
is multiplicative, and its series is absolutely convergent. Indeed, using
\[
\phi(\ell)\ll \ell,
\qquad
\phi_2(\ell)\geq\frac{\ell^2}{\tau(\ell)},
\]
we have
\[
\left|
\frac{\mu(\ell)\phi(\ell)}
{\ell\phi_2(\ell)}
\right|
\ll
\frac{\tau(\ell)}{\ell^2}.
\]
Consequently,
\[
\sum_{\ell\geq1}
\frac{\mu(\ell)\phi(\ell)}
{\ell\phi_2(\ell)}
=
\prod_p
\left(
1-\frac{1}{p(p+1)}
\right).
\]
Moreover, by~\eqref{WS1},
\[
\sum_{\ell>x}
\left|
\frac{\mu(\ell)\phi(\ell)}
{\ell\phi_2(\ell)}
\right|
\ll
\sum_{\ell>x}\frac{\tau(\ell)}{\ell^2}
\ll
\frac{\log x}{x}.
\]
It follows that
\[
\sum_{\ell\leq x}
\frac{\mu(\ell)\phi(\ell)}
{\ell\phi_2(\ell)}
=
\prod_p
\left(
1-\frac{1}{p(p+1)}
\right)
+
\Ocal\!\left(\frac{\log x}{x}\right).
\]
Furthermore,~\eqref{taustar} gives
\[
\log x
\sum_{\ell\leq x}
\frac{\tau^{\star}(\ell)}{\ell}
\ll
(\log x)^3.
\]
Substitution into~\eqref{WA} therefore yields
\begin{equation}
\label{L0-mu}
L_0(x)
=
\frac{x}{\zeta(2)}
\prod_p
\left(
1-\frac{1}{p(p+1)}
\right)
+
\Ocal\!\left((\log x)^3\right).
\end{equation}

We next consider $m=1$. In this case,
\[
L_1(x)
=
\sum_{\ell\leq x}
\frac{\mu(\ell)}{\ell}
\sum_{\substack{d\leq x/\ell\\ \gcd(d,\ell)=1}}
\frac{\phi_{-1}(d)}{d}.
\]
By~\eqref{eqlem16}, with $t=\ell$ and $x$ replaced by $x/\ell$, we have
\begin{equation*}
|L_1(x)|
\ll{}
\sum_{\ell\leq x}
\frac{\phi(\ell)}{\ell^2}
\left(\log\frac{x}{\ell}\right)^2
+
\sum_{\ell\leq x}
\frac{\phi(\ell)}{\ell^2}
\log\log(3\ell)
\log\frac{x}{\ell}
+
\sum_{\ell\leq x}
\frac{\tau(\ell)}{\ell}
\log\frac{x}{\ell}.
\end{equation*}
Since
\[
\sum_{\ell\leq x}\frac{\phi(\ell)}{\ell^2}
\ll\log x
\]
by~\eqref{eqlem26}, and
\[
\sum_{\ell\leq x}\frac{\tau(\ell)}{\ell}
\ll(\log x)^2
\]
by~\eqref{eqlem24}, it follows that
\begin{equation}
\label{L1-mu}
L_1(x)
=
\Ocal\!\left((\log x)^3\right).
\end{equation}

Finally, let $m\geq2$. Reversing the order of summation and applying
\eqref{eqlem15} with $s=2m-1$ gives
\begin{equation*}
|L_m(x)|
\ll_m{}
\sum_{\ell\leq x}
\frac{\phi_{2m-1}(\ell)}{\ell^{2m}}
\log\frac{x}{\ell}+\sum_{\ell\leq x}
\frac{\sigma(\ell)}{\ell^2}
\log\frac{x}{\ell}.
\end{equation*}
Partial summation, together with the classical estimates
\[
\sum_{n\leq y}\frac{\phi_s(n)}{n}
=
\frac{y^s}{s\zeta(s+1)}
+
\Ocal_s(y^{s-1}),
\qquad s\geq2,
\]
and
\[
\sum_{n\leq y}\frac{\sigma(n)}{n}
=
\zeta(2)y+\Ocal(\log y),
\]
implies that
\[
\sum_{\ell\leq x}
\frac{\phi_{2m-1}(\ell)}{\ell^{2m}}
\ll_m\log x
\]
and
\[
\sum_{\ell\leq x}
\frac{\sigma(\ell)}{\ell^2}
\ll\log x.
\]
Consequently,
\begin{equation}
\label{Lm-mu}
L_m(x)
=
\Ocal_m\!\left((\log x)^2\right),
\qquad m\geq2.
\end{equation}

Combining~\eqref{Umu-decomposition},
\eqref{Umu-first-term}, \eqref{L0-mu},
\eqref{L1-mu}, and \eqref{Lm-mu}, we conclude that
\[
U_{\mu}(x)
=
\frac{x}{(r+1)\zeta(2)}
\prod_p
\left(
1-\frac{1}{p(p+1)}
\right)
+
\Ocal_r\!\left((\log x)^3\right).
\]
This proves~\eqref{eqcor12}.\\

We now prove~\eqref{eqcor13}. Taking $f=\tau$ in~\eqref{eq1}, we obtain
\begin{equation}
\label{Utau-decomposition}
U_{\tau}(x)
=
\frac12\sum_{n\leq x}\frac{\tau(n)}{n}
+
\frac{1}{r+1}
\sum_{m=0}^{\lfloor r/2\rfloor}
\binom{r+1}{2m}B_{2m}G_m(x),
\end{equation}
where
\[
G_m(x)
:=
\sum_{\substack{d\ell\leq x\\ \gcd(d,\ell)=1}}
\frac{\phi_{1-2m}(d)}{d}
\frac{\tau(\ell)}{\ell}.
\]
By~\eqref{eqlem24},
\begin{equation}
\label{Utau-first-term}
\sum_{n\leq x}\frac{\tau(n)}{n}
\ll
(\log x)^2.
\end{equation}
We first consider the case $m=0$. Reversing the order of summation gives
\[
G_0(x)
=
\sum_{\ell\leq x}
\frac{\tau(\ell)}{\ell}
\sum_{\substack{d\leq x/\ell\\ \gcd(d,\ell)=1}}
\frac{\phi(d)}{d}.
\]
Applying~\eqref{eqlem14} with $t=\ell$ and with $x$ replaced by
$x/\ell$, we obtain
\begin{equation}
\label{eqproof51}
G_0(x)
=
\frac{x}{\zeta(2)}
\sum_{\ell\leq x}
\frac{\tau(\ell)\phi(\ell)}
{\ell\phi_2(\ell)}+
\Ocal\!\left(
\log x
\sum_{\ell\leq x}
\frac{\tau(\ell)\tau^\star(\ell)}{\ell}
\right).
\end{equation}
The function $\frac{\tau(\ell)\phi(\ell)}
{\ell\phi_2(\ell)}$ is multiplicative, and its associated series is absolutely convergent.
Indeed, using
\[
\phi(\ell)\ll\ell,
\qquad
\phi_2(\ell)\geq\frac{\ell^2}{\tau(\ell)},
\]
we obtain
\[
\frac{\tau(\ell)\phi(\ell)}
{\ell\phi_2(\ell)}
\ll
\frac{\tau^2(\ell)}{\ell^2}.
\]
Hence, by~\eqref{WS2},
\[
\sum_{\ell>x}
\frac{\tau(\ell)\phi(\ell)}
{\ell\phi_2(\ell)}
\ll
\frac{(\log x)^3}{x}.
\]
A calculation of the local factors gives
\[
\sum_{\ell\geq1}
\frac{\tau(\ell)\phi(\ell)}
{\ell\phi_2(\ell)}
=
\prod_p
\left(
1+
\frac{p(2p^2-1)}
{(p-1)^2(p+1)^3}
\right).
\]
Therefore,
\begin{equation}
\label{G0-main-series}
\sum_{\ell\leq x}
\frac{\tau(\ell)\phi(\ell)}
{\ell\phi_2(\ell)}
=
\prod_p
\left(
1+
\frac{p(2p^2-1)}
{(p-1)^2(p+1)^3}
\right)
+
\Ocal\!\left(
\frac{(\log x)^3}{x}
\right).
\end{equation}

Since $\tau^\star(\ell)\leq\tau(\ell)$, we have
\[
\log x
\sum_{\ell\leq x}
\frac{\tau(\ell)\tau^\star(\ell)}{\ell}
\leq
\log x
\sum_{\ell\leq x}
\frac{\tau^2(\ell)}{\ell}
\ll
(\log x)^5
\]
by~\eqref{WS3}. Combining this estimate with
\eqref{eqproof51} and~\eqref{G0-main-series}, we obtain
\begin{equation}
\label{G0-tau}
G_0(x)
=
\frac{x}{\zeta(2)}
\prod_p
\left(
1+
\frac{p(2p^2-1)}
{(p-1)^2(p+1)^3}
\right)
+
\Ocal\!\left((\log x)^5\right).
\end{equation}

We next consider $m=1$. In this case,
\[
G_1(x)
=
\sum_{\ell\leq x}
\frac{\tau(\ell)}{\ell}
\sum_{\substack{d\leq x/\ell\\ \gcd(d,\ell)=1}}
\frac{\phi_{-1}(d)}{d}.
\]
Applying~\eqref{eqlem16} with $t=\ell$ and $x$ replaced by
$x/\ell$, we find that
\begin{align*}
|G_1(x)|
\ll
\sum_{\ell\leq x}
\frac{\tau(\ell)\phi(\ell)}{\ell^2}
\left(\log\frac{x}{\ell}\right)^2
+
\sum_{\ell\leq x}
\frac{\tau(\ell)\phi(\ell)}{\ell^2}
\log\log(3\ell)
\log\frac{x}{\ell}
+
\sum_{\ell\leq x}
\frac{\tau^2(\ell)}{\ell}
\log\frac{x}{\ell}.
\end{align*}
Since $\phi(\ell)\leq\ell$, the first two sums are bounded by
\[
\ll
(\log x)^2
\sum_{\ell\leq x}\frac{\tau(\ell)}{\ell}
+
\log x\log\log(3x)
\sum_{\ell\leq x}\frac{\tau(\ell)}{\ell}
\ll
(\log x)^4\log\log(3x).
\]
The last sum satisfies
\[
\sum_{\ell\leq x}
\frac{\tau^2(\ell)}{\ell}
\log\frac{x}{\ell}
\ll
\log x
\sum_{\ell\leq x}
\frac{\tau^2(\ell)}{\ell}
\ll
(\log x)^5
\]
by~\eqref{WS3}. Consequently,
\begin{equation}
\label{G1-tau}
G_1(x)
=
\Ocal\!\left((\log x)^5\right).
\end{equation}

Finally, let $m\geq2$. By reversing the order of summation and applying
\eqref{eqlem15} with $s=2m-1$, we obtain
\begin{equation*}
|G_m(x)|
\ll_m
\sum_{\ell\leq x}
\frac{\tau(\ell)\phi_{2m-1}(\ell)}
{\ell^{2m}}
\log\frac{x}{\ell}
+
\sum_{\ell\leq x}
\frac{\tau(\ell)\sigma(\ell)}
{\ell^2}
\log\frac{x}{\ell}.
\end{equation*}
Since
\[
\phi_{2m-1}(\ell)\leq\ell^{2m-1},
\]
the first sum is bounded by
\[
\ll
\log x
\sum_{\ell\leq x}
\frac{\tau(\ell)}{\ell}
\ll
(\log x)^3.
\]
Furthermore, the inequality
\[
\sigma(\ell)+\phi(\ell)\leq\ell\tau(\ell),
\qquad \ell\geq2,
\]
implies
\[
\sigma(\ell)\leq\ell\tau(\ell).
\]
Therefore,
\begin{align*}
\sum_{\ell\leq x}
\frac{\tau(\ell)\sigma(\ell)}
{\ell^2}
\log\frac{x}{\ell}
&\ll
\log x
\sum_{\ell\leq x}
\frac{\tau^2(\ell)}{\ell}
\\
&\ll
(\log x)^5
\end{align*}
by~\eqref{WS3}. It follows that
\begin{equation}
\label{Gm-tau}
G_m(x)
=
\Ocal_m\!\left((\log x)^5\right),
\qquad m\geq2.
\end{equation}

Combining~\eqref{Utau-decomposition},
\eqref{Utau-first-term}, \eqref{G0-tau},
\eqref{G1-tau}, and \eqref{Gm-tau}, we conclude that
\[
U_{\tau}(x)
=
\frac{x}{(r+1)\zeta(2)}
\prod_p
\left(
1+
\frac{p(2p^2-1)}
{(p-1)^2(p+1)^3}
\right)
+
\Ocal_r\!\left((\log x)^5\right).
\]
This proves~\eqref{eqcor13}.\\

We finally prove~\eqref{eqcor14}. Taking $f=\phi_2$ in~\eqref{eq2}, and
using the multiplicativity of $\phi_2$ together with $(d,\ell)=1$, we obtain
\[
\frac{\phi_2(\ell)}{\phi_2(d\ell)}
=
\frac{1}{\phi_2(d)}.
\]
Consequently,
\begin{equation}
\label{Vphi2-decomposition}
V_{\phi_2}(x)
=
\frac{x}{2}
-\frac{\theta(x)}{2}
-\frac14
+
\frac{1}{r+1}
\sum_{m=0}^{\lfloor r/2\rfloor}
\binom{r+1}{2m}
B_{2m}F_m(x),
\end{equation}
where
\[
F_m(x)
:=
\sum_{\substack{d\ell\leq x\\ \gcd(d,\ell)=1}}
\frac{\phi_{1-2m}(d)}{\phi_2(d)}.
\]

We first consider the case $m=0$. Reversing the order of summation gives
\[
F_0(x)
=
\sum_{d\leq x}
\frac{\phi(d)}{\phi_2(d)}
\sum_{\substack{\ell\leq x/d\\ \gcd(\ell,d)=1}}1.
\]
Applying~\eqref{eqlem11}, we obtain
\begin{equation}
\label{eqproof41}
F_0(x)
=
x
\sum_{d\leq x}
\frac{\phi(d)^2}{d^2\phi_2(d)}
-
\sum_{d\leq x}
\frac{\phi(d)}{\phi_2(d)}
\sum_{k\mid d}
\mu(k)
\theta\!\left(\frac{x}{dk}\right).
\end{equation}
The arithmetic function $\frac{\phi(d)^2}{d^2\phi_2(d)}$ is multiplicative, and its series is absolutely convergent. A calculation
of the local factors yields
\[
\sum_{d\geq1}
\frac{\phi(d)^2}{d^2\phi_2(d)}
=
\prod_p
\left(
1+\frac{1}{(p+1)^2}
\right).
\]
Moreover, since $\phi(d)^2\leq\phi_2(d)$,
\[
\sum_{d>x}
\frac{\phi(d)^2}{d^2\phi_2(d)}
\leq
\sum_{d>x}\frac1{d^2}
\ll
\frac1x.
\]
Therefore,
\begin{equation}
\label{F0-main-series}
\sum_{d\leq x}
\frac{\phi(d)^2}{d^2\phi_2(d)}
=
\prod_p
\left(
1+\frac{1}{(p+1)^2}
\right)
+
\Ocal\!\left(\frac1x\right).
\end{equation}

For the second term in~\eqref{eqproof41}, the bound
$|\theta(y)|\leq 1/2$ gives
\begin{align*}
\left|
\sum_{d\leq x}
\frac{\phi(d)}{\phi_2(d)}
\sum_{k\mid d}
\mu(k)
\theta\!\left(\frac{x}{dk}\right)
\right|
&\ll
\sum_{d\leq x}
\frac{\phi(d)\tau^\star(d)}{\phi_2(d)}
\\
&\ll
\sum_{d\leq x}
\frac{\phi(d)\tau^2(d)}{d^2}
\\
&\ll
\sum_{d\leq x}
\frac{\tau^2(d)}{d}
\ll
(\log x)^4.
\end{align*}
Here we used $\tau^\star(d)\leq\tau(d)$,
\[
\phi_2(d)\geq\frac{d^2}{\tau(d)},
\]
and~\eqref{WS3}. Combining this estimate with
\eqref{eqproof41} and~\eqref{F0-main-series}, we obtain
\begin{equation}
\label{eqproof42}
F_0(x)
=
x
\prod_p
\left(
1+\frac{1}{(p+1)^2}
\right)
+
\Ocal\!\left((\log x)^4\right).
\end{equation}

We now consider $m\geq1$. As above,~\eqref{eqlem11} gives
\begin{equation}
\label{eqproof43}
F_m(x)
=
x
\sum_{d\leq x}
\frac{\phi_{1-2m}(d)\phi(d)}
{d^2\phi_2(d)}
-
\sum_{d\leq x}
\frac{\phi_{1-2m}(d)}{\phi_2(d)}
\sum_{k\mid d}
\mu(k)
\theta\!\left(\frac{x}{dk}\right).
\end{equation}
The function $
\frac{\phi_{1-2m}(d)\phi(d)}
{d^2\phi_2(d)}$ is multiplicative, and the corresponding series is absolutely convergent. Its Euler product is
\[
\sum_{d\geq1}
\frac{\phi_{1-2m}(d)\phi(d)}
{d^2\phi_2(d)}
=
\prod_p
\left(
1-
\frac{p(p^{2m-1}-1)}
{(p+1)(p^{2m+2}-1)}
\right).
\]
Since
\[
\left|\phi_{1-2m}(d)\right|\leq1
\qquad (m\geq1),
\]
and
\[
\frac{\phi(d)}{\phi_2(d)}
\ll
\frac{\tau(d)}{d},
\]
we have
\[
\sum_{d>x}
\left|
\frac{\phi_{1-2m}(d)\phi(d)}
{d^2\phi_2(d)}
\right|
\ll
\sum_{d>x}\frac{\tau(d)}{d^3}
\ll
\frac{\log x}{x^2}
\]
by~\eqref{WS4}. It follows that
\begin{equation}
\label{Fm-main-series}
\sum_{d\leq x}
\frac{\phi_{1-2m}(d)\phi(d)}
{d^2\phi_2(d)}
=
\prod_p
\left(
1-
\frac{p(p^{2m-1}-1)}
{(p+1)(p^{2m+2}-1)}
\right)
+
\Ocal_m\!\left(\frac{\log x}{x^2}\right).
\end{equation}

For the second term in~\eqref{eqproof43}, we use
$|\theta(y)|\leq1/2$, $|\phi_{1-2m}(d)|\leq1$, and
$\tau^\star(d)\leq\tau(d)$ to obtain
\begin{align*}
\left|
\sum_{d\leq x}
\frac{\phi_{1-2m}(d)}{\phi_2(d)}
\sum_{k\mid d}
\mu(k)
\theta\!\left(\frac{x}{dk}\right)
\right|
&\ll
\sum_{d\leq x}
\frac{\tau^\star(d)}{\phi_2(d)}
\\
&\ll
\sum_{d\leq x}
\frac{\tau^2(d)}{d^2}
\ll 1.
\end{align*}
Thus, by~\eqref{eqproof43} and~\eqref{Fm-main-series},
\begin{equation}
\label{eqproof44}
F_m(x)
=
x
\prod_p
\left(
1-
\frac{p(p^{2m-1}-1)}
{(p+1)(p^{2m+2}-1)}
\right)
+
\Ocal_m(1),
\qquad m\geq1.
\end{equation}

Finally, substituting~\eqref{eqproof42} and~\eqref{eqproof44}
into~\eqref{Vphi2-decomposition}, and observing that
$\theta(x)=\Ocal(1)$, we conclude that
\begin{multline*}
V_{\phi_2}(x)
=
\left(
\frac12
+
\frac{1}{r+1}
\prod_p
\left(
1+\frac{1}{(p+1)^2}
\right)
\right)x
\\
+
\frac{x}{r+1}
\sum_{m=1}^{\lfloor r/2\rfloor}
\binom{r+1}{2m}
B_{2m}
\prod_p
\left(
1-
\frac{p(p^{2m-1}-1)}
{(p+1)(p^{2m+2}-1)}
\right)
+
\Ocal_r\!\left((\log x)^4\right).
\end{multline*}
This proves~\eqref{eqcor14}.
\section*{Acknowledgement}
The authors sincerely thank Isao Kiuchi and Kohji Matsumoto for their careful reading of the manuscript and for their valuable comments and suggestions.

\medskip\noindent {Waseem ALASS: 
Johannes Kepler University Linz, Altenbergerstrasse 69, 4040 Linz, Austria. E-mail: {\tt waseem.alass@jku.at}}

\medskip\noindent {Sumaia Saad Eddin: 
Johann Radon Institute for Computational and Applied Mathematics, Austrian Academy of Sciences, Altenbergerstrasse 69, 4040 Linz, Austria.}\\
E-mail: {\tt sumaia.saad-eddin@ricam.oeaw.ac.at}}

\end{document}